\documentclass{amsart}
\usepackage{setspace}

\usepackage{amsthm}
\usepackage{a4}
\usepackage{amsfonts}
\usepackage{amsmath}
\usepackage{amssymb}
\usepackage{mathrsfs}
\usepackage{mathtools}

\usepackage{enumerate}

\usepackage{color}

\parskip\medskipamount
\DeclareFontFamily{OT1}{manual}{}
\DeclareFontShape{OT1}{manual}{m}{n}{ <10> manfnt }{}

\makeatletter
\def\blfootnote{\xdef\@thefnmark{}\@footnotetext}
\makeatother

\newcommand*{\ox}{\otimes}

\DeclarePairedDelimiter\abs{\lvert}{\rvert}

\newcommand*{\sgn}{\mathrm{sgn}}

\newcommand*{\<}{\langle}
\renewcommand*{\>}{\rangle}
\newcommand*{\x}{\times}

\newcommand{\openquat}[3]{(#1,#2)_{#3}}

\newcommand{\leqs}{\leqslant}
\newcommand{\geqs}{\geqslant}

\newtheorem{lemma}{Lemma}[section]
\newtheorem{theorem}[lemma]{Theorem}
\newtheorem{prop}[lemma]{Proposition}
\newtheorem{cor}[lemma]{Corollary}

\newtheorem{question}[lemma]{Question}

\theoremstyle{definition}

\theoremstyle{remark}
\newtheorem{remark}[lemma]{Remark}

\theoremstyle{definition}
\newtheorem{example}[lemma]{Example}

\title{Multiples of Pfister forms}
\author{James O'Shea}
\date{}
\begin{document}

%
%

\maketitle



\begin{abstract}\noindent The isotropy of multiples of Pfister forms is studied. In particular, an improved lower bound on the values of their first Witt indices is obtained. A number of corollaries of this result are outlined. An investigation of generic Pfister multiples is also undertaken. These results are applied to distinguish between properties preserved by Pfister products.
\end{abstract}

\blfootnote{James O'Shea,\\ National University of Ireland Maynooth and University College Dublin.\\ \emph{E-mail}: james.oshea@maths.nuim.ie or james.oshea@ucd.ie}



\section{Introduction}

Given the centrality of Pfister forms to the theory of quadratic forms, the isotropy of their multiples has been a topic of long-standing interest. A classical result in this regard, established by Elman and Lam in the early seventies, states that, for $\pi$ a Pfister form and $q$ an arbitrary form, the Witt index of their product $\pi\otimes q$ is a multiple of the dimension of $\pi$. Hence, given an anisotropic product over some ground field, one can view this product over a generic field extension that makes it isotropic, thereby obtaining that the first Witt index of $\pi\otimes q$ is a multiple of the dimension of $\pi$, and thus at least the dimension of $\pi$, a bound which is regularly invoked in the literature. 

Somewhat surprisingly, it is possible to say more regarding the first Witt index of Pfister multiples. The main result of this article, Theorem~\ref{i1bound}, states that the first Witt index of $\pi\otimes q$ is at least the first Witt index of $q$ times the dimension of $\pi$, thereby establishing an improved lower bound on the value of the Witt index of the product over every extension that makes it isotropic, and raising the possibility that more can be said regarding the other higher Witt indices of Pfister multiples. Whereas the value of the first Witt index of a Pfister multiple can exceed this bound (see Example~\ref{nocon}), we can establish conditions on a form which ensure that the bound is attained by its Pfister products (Proposition~\ref{pivalues} and Proposition~\ref{i1maxspl}). As a consequence, we can add the maximal-splitting property to the list of form-theoretic properties known to be preserved by Pfister products (Corollary~\ref{maxsplcor}). Certain of these results have been applied and referenced (without proof) in \cite{OS1}.




While many phenomena are preserved under multiplication by Pfister forms, one should not expect a correspondence between the properties of a form and those of its Pfister multiples. However, it seems reasonable to suggest that such a correspondence may hold with respect to multiplication by ``generic Pfister forms'', those generated by transcendental elements. Section 3 contains a number of results which support this view, with Proposition~\ref{pntrans}, for example, establishing that a form is a Pfister neighbour if and only if its generic Pfister multiples are Pfister neighbours.

%


Such results are of relevance to the task of distinguishing between properties preserved by Pfister products, as they provide a framework for extending examples existing in low dimensions. We conclude with a discussion of some apposite open questions, offering a generalisation of Hoffmann's construction of forms with maximal splitting that are not Pfister neighbours and an extension of Vishik's recent example of a $16$-dimensional form with first Witt index equal to two that is not a Pfister multiple.

Henceforth, we will let $F$ denote a field of characteristic different from two. The term ``form'' will refer to a regular quadratic form. Every form over $F$ can be diagonalised. Given $a_1,\ldots ,a_n\in F^{\x}$ for $n\in\mathbb{N}$, we denote by $\<a_1,\ldots,a_n\>$ the $n$-dimensional quadratic form $a_1X_1^2+\ldots+a_nX_n^2$. If $p$ and $q$ are forms over $F$, we denote by $p\perp q$ their orthogonal sum and by $p\otimes q$ their tensor product. For $n\in\mathbb{N}$, we will denote the orthogonal sum of $n$ copies of $q$ by $n\x q$. We use $aq$ to denote $\<a\>\otimes q$ for $a\in F^{\x}$. We write $p\simeq q$ to indicate that $p$ and $q$ are isometric, and say that $p$ and $q$ are \emph{similar} (over $F$) if $p\simeq aq$ for some $a\in F^{\x}$. For $q$ a form over $F$ and $K/F$ a field extension, we will employ the notation $q_K$ when viewing $q$ as a form over $K$ via the canonical embedding. A form $p$ is a \emph{subform of $q$} if $q\simeq p\perp r$ for some form $r$, in which case we will write $p\subset q$. A form $q$ \emph{represents $a\in F$} if there exists a vector $v$ such that $q(v)=a$. We denote by $D_F(q)$ the set of values in $F^{\x}$ represented by $q$. A form over $F$ is \emph{isotropic} if it represents zero non-trivially, and \emph{anisotropic} otherwise. Every form $q$ has a decomposition $q\simeq q_{\mathrm{an}}\perp i(q)\x\<1,-1\>$ where the anisotropic form $q_{\mathrm{an}}$ and the integer $i(q)$, the \emph{Witt index of $q$}, are uniquely determined. A form $q$ is \emph{hyperbolic} if $q_{\mathrm{an}}$ is trivial, whereby $i(q)=\frac 1 2 \dim q$. Two anisotropic forms $p$ and $q$ over $F$ are \emph{isotropy equivalent} if for every field extension $K/F$ we have that $p_K$ is isotropic if and only if $q_K$ is isotropic. The following basic fact (see \cite[Exercise I.$16$]{LAM}) will be employed frequently. 
\begin{lemma}\label{star} If $p\subset q$ with $\dim p\geqs\dim q-i(q)+1$, then $p$ is isotropic.\end{lemma}








%

If $q$ is an even-dimensional form, its \emph{Clifford invariant} is $[C(q)]$, the class of the Clifford algebra of $q$ in the Brauer group of $F$. The Clifford invariant of an odd-dimensional form $q$ is $[C_0(q)]$, the Brauer class of the even Clifford algebra of $q$ (the subalgebra of elements of even degree in $C(q)$). Formulae for the computation of Clifford invariants can be found in \cite[Chapter V, (3.13)]{LAM}. The \emph{Schur index of a central simple algebra} is the square root of the dimension of a Brauer-equivalent division algebra. An \emph{ordering of $F$} is a set $P\subset F^{\x}$ such that $P\cup -P=F^{\x}$ and $x+y, xy\in P$ for all $x,y\in P$. We say that $F$ is a \emph{(formally) real field} if it has an ordering. Given a form $q$ over $F$ and an ordering $P$ of $F$, the \emph{signature of $q$ at $P$}, denoted $\sgn_P(q)$, is the number of coefficients in a diagonalisation of $q$ that are in $P$ minus the number that are not in $P$. A form $q$ over $F$ is \emph{indefinite at $P$} if $|\sgn_P(q)|<\dim q$. 









For $n\in\mathbb{N}$, an \emph{$n$-fold Pfister form} over $F$ is a form isometric to $\<1,a_1\>\otimes\ldots\otimes\<1,a_n\>$ for some $a_1,\ldots ,a_n\in F^{\times}$ (the form $\< 1\>$ is the $0$-fold Pfister form). Isotropic Pfister forms are hyperbolic \cite[Theorem X.1.7]{LAM}. A form $\tau$ over $F$ is a \emph{neighbour} of a Pfister form $\pi$ if $\tau\subset a\pi$ for some $a\in F^{\x}$ and $\dim{\tau}>\frac 1 2 \dim\pi$. For $\tau$ a neighbour of a Pfister form $\pi$ with $\tau\perp \gamma\simeq a\pi$ for some $a\in F^{\x}$, the form $\gamma$ is called the \emph{complementary form} of $\tau$. All forms of dimension not greater than one are said to be \emph{excellent}; a form $q$ of dimension $n\geqs 2$ is \emph{excellent} if $q$ is a Pfister neighbour and the complementary form of $q$ is excellent. A form $q$ over $F$ is \emph{round} if $D_F(q)=G_F(q)$, where $G_F(q)=\{a\in F^{\x}\mid aq\simeq q\}$ is the group of similarity factors of $q$. Pfister forms are round (see \cite[Theorem X.1.8]{LAM}). 






For a form $q$ over $F$ with $\dim q=n\geqs 2$ and $q\not\simeq\<1,-1\>$, the \emph{function field $F(q)$ of $q$} is the quotient field of the integral domain $F[X_1,\ldots ,X_n]/(q(X_1,\ldots ,X_n))$ (this is the function field of the affine quadric $q(X)=0$ over $F$). To avoid case distinctions, we set $F(q)=F$ if $\dim q\leqs 1$ or $q\simeq\<1,-1\>$. Letting $F_0=F$, $i_0(q)=i(q)$ and $q_0\simeq q_{\mathrm{an}}$, following Knebusch \cite{Kn1} we inductively define $$F_{j+1}=F_{j}(q_{j}), \quad i_{j+1}(q)=i((q_j)_{F_{j+1}})\quad \text{and}\quad q_{j+1}\simeq((q_{j})_{F_{j+1}})_{\mathrm{an}},$$ stopping when $\dim q_h\leqs 1$. This integer $h$ is the \emph{height of $q$}, the tower of fields $F=F_0\subset F_1\subset\ldots\subset F_h$ is the \emph{generic splitting tower of $q$}, the forms $q_1,\ldots ,q_h$ are the \emph{higher kernel forms of $q$} and the natural numbers $i_1(q),\ldots , i_h(q)$ are the \emph{higher Witt indices of $q$}. The sequence $(i_1(q),\ldots , i_h(q))$ is called the {\em (incremental) splitting pattern of $q$}. For all forms $p$ over $F$ and all extensions $K/F$ such that $q_K$ is isotropic, we have that $i(p_{F(q)})\leqs i(p_K)$ (see \cite[Proposition 3.1 and Theorem 3.3]{Kn1}). In particular, with respect to $i_1(q)$, the {\em first Witt index of $q$}, we have that $i_1(q)\leqs i(q_K)$ for all extensions $K/F$ such that $q_K$ is isotropic. An anisotropic form $q$ is said to have \emph{maximal splitting} if $\dim q-i_1(q)$ is a power of two. As per \cite[Theorem X.$4.1$]{LAM}, $F(q)$ is a purely-transcendental extension of $F$ if and only if $q$ is isotropic over $F$. On account of this fact, one can see that two anisotropic forms $p$ and $q$ over $F$ are isotropy equivalent if and only if $p_{F(q)}$ and $q_{F(p)}$ are isotropic. The behaviour of orderings with respect to function field extensions is governed by the following result due to Elman, Lam and Wadsworth \cite[Theorem 3.5]{ELW} and, independently, Knebusch \cite[Lemma 10]{GS}.
\begin{theorem}\label{ELW} Let $q$ be a form of dimension at least two over a real field $F$. An ordering $P$ of $F$ extends to $F(q)$ if and only if $q$ is indefinite at $P$.
\end{theorem}

\cite[Theorem 1]{H} and \cite[Theorem 4.1]{KM} represent important isotropy criteria with respect to function fields of quadratic forms. We recall these results below.



\begin{theorem}\label{H95} $($Hoffmann$)$ Let $p$ and $q$ be forms over $F$ such that $p$ is anisotropic. If $\dim p\leqs 2^n<\dim q$ for some integer $n\geqs 0$, then
$p_{F(q)}$ is anisotropic.\end{theorem}

%
%

\begin{theorem}\label{km} $($Karpenko, Merkurjev$)$ Let $p$ and $q$ be anisotropic forms over $F$ such that $p_{F(q)}$ is isotropic. Then \begin{enumerate}[$(i)$]
\item $\dim p - i_1(p)\geqs\dim q - i_1(q)$;
\smallskip
\item $\dim p - i_1(p)=\dim q - i_1(q)$ if and only if $q_{F(p)}$ is isotropic.\end{enumerate}\end{theorem}




%

%

%
%




Over $F(\!(x)\!)$, the Laurent series field in the variable $x$ over $F$, we recall that every non-zero square class can be represented by $a$ or $ax$ for some $a\in F^{\x}$, whereby every form $\varphi$ over $F(\!(x)\!)$ can be written as $p\perp xq$ for $p$ and $q$ forms over $F$. We recall the following folkloric result regarding forms over Laurent series fields.





\begin{lemma}\label{Hlemma} Let $p$ and $q$ be forms over $F$. Considering $p\perp x q$ as a form over $F(\!(x)\!)$, we have that $i(p\perp x q)=i(p)+ i(q)$.
\end{lemma}


\begin{proof} Applying Springer's Theorem for complete discretely valued fields \cite[Theorem VI.$1.4$]{LAM}, one obtains that $p\perp x q$ is anisotropic over $F(\!(x)\!)$ if and only if $p$ and $q$ are anisotropic over $F$. The result follows by applying Witt decomposition to the forms $p$ and $q$ over $F$.
\end{proof}

\section{The isotropy of multiples of Pfister forms}



Since the isotropy of scalar multiples of Pfister forms is well understood (indeed, an anisotropic form $q$ of dimension at least two is a scalar multiple of a Pfister form if and only if $q$ is hyperbolic over $F(q)$, see \cite[Corollary $23.4$]{EKM}), we will restrict our attention to multiples of Pfister forms with forms of dimension at least two. 


Elman and Lam obtained a number of important results on the isotropy of multiples of Pfister forms in the early seventies. The following classical result, as formulated below, is a consequence of their representation theorem \cite[Theorem 1.4]{EL} (see \cite[Lemma 3.1]{H4} for a proof of this). Wadsworth and Shapiro \cite[Theorem 2]{WS} established that this result holds, more generally, for multiples of round forms.


%

\begin{theorem}\label{WS} $($Elman, Lam$)$ Let $\pi$ be an anisotropic Pfister form over F and let $q$ be a form over $F$ of dimension at least two. 
If $\pi\otimes q$ is isotropic, then there exist forms $q_1$ and $q_2$ over $F$ such that $\pi\otimes q_1$ is anisotropic, $q_2$ is hyperbolic,
and $\pi\otimes q\simeq \pi\otimes q_1\perp \pi\otimes q_2$. In particular, $i(\pi\otimes q)=(\dim\pi) i(q_2)$.
\end{theorem}

With respect to the above theorem, we clearly have that $i(q_2)\geqs i(q)$. These quantities do not appear to satisfy any stronger relation in general however (indeed, the form $q$ may be anisotropic).

Theorem~\ref{WS} has a number of important consequences. The following statement, which is regularly applied in the literature, is one such result.


\begin{cor}\label{WScor} Let $q$ a form of dimension at least two and $\pi$ similar to a Pfister form be such that $\pi\otimes q$ is anisotropic over $F$. Then $i_1(\pi\otimes q)\geqs \dim\pi$.
\end{cor}


Thus, we have that $i((\pi\otimes q)_K)\geqs \dim \pi$ for $K/F$ such that $\pi\otimes q$ is isotropic over $K$. Moreover, as $F(\pi\otimes q)$ is a generic zero field of $\pi\otimes q$, it is often the case that when $i_1(\pi\otimes q)$ can be precisely determined, its value actually equals $\dim\pi$. Thus, while the value of $i_1(\pi\otimes q)$ has long been known to be a multiple of $\dim\pi$, in accordance with Theorem~\ref{WS}, it is perhaps surprising that more can be said regarding this multiple in general. We will invoke \cite[Th\'eor\`eme 6.4.2]{R}, stated below, to achieve this. In his thesis, Roussey offers a number of proofs of this result, which he introduces as being already known but hitherto unwritten.


%
%


%
%




\begin{theorem}\label{R} $($Roussey$)$ Let $p$ and $q$ be two forms over F of dimension at least two and let $\pi$ be similar to a Pfister form over $F$. If $p$ is isotropic over $F(q)$, then $\pi\otimes p$ is isotropic over $F(\pi\otimes q)$. 
\end{theorem}




With regard to Theorem~\ref{R}, we note that the corresponding statement with respect to hyperbolicity also holds, having been established by Fitzgerald \cite[Theorem $3.2$]{F2}.

Our opening result, concerning products of $\pi$ with the higher kernel forms of $q$, establishes the aforementioned refinement of Corollary~\ref{WScor}.

\begin{theorem}\label{i1bound} Let $\pi$ be similar to an anisotropic Pfister form over $F$. Let $q$ be an anisotropic form over $F$ of dimension at least two. Let $F=F_0\subset F_1\subset \ldots \subset F_h$ denote the generic splitting tower of $q$ and $q\simeq q_0, q_1\ldots ,q_h$ the kernel forms of $q$. Suppose that $\pi\otimes q_j$ is anisotropic over $F_j$ for some fixed $j$ satisfying $0\leqs j\leqs h-1$. Then $i_1(\pi\otimes q_j)\geqs (\dim\pi)i_{j+1}(q)$.

In particular, if $\pi\otimes q$ is anisotropic over $F$, then $i_1(\pi\otimes q)\geqs (\dim\pi)i_1(q)$.
\end{theorem}

\begin{proof} We consider the anisotropic form $\pi\otimes q_j$ over $F_j$ for $j$ such that $0\leqs j\leqs h-1$. If $i_{j+1}(q)=1$, then the statement follows immediately from Corollary~\ref{WScor}. Hence, we may assume that $i_{j+1}(q) > 1$. Let $q'\subset q_j$ over $F_j$ of dimension $\dim q_j -i_1(q_j)+1$, whereby $q'$ is a proper subform of $q_j$ as $i_{j+1}(q)=i_1(q_j)$. Lemma~\ref{star} implies that $q'$ is isotropic over $F_j(q_j)$. Hence, $\pi\otimes q'$ is isotropic over $F_j(\pi\otimes q_j)$ by Theorem~\ref{R}. As $\pi\otimes q'\subset\pi\otimes q_j$, we have that $\pi\otimes q'$ is anisotropic over $F_j$ from our assumption and, furthermore, that $\pi\otimes q_j$ is isotropic over $F_j(\pi\otimes q')$, whereby $\pi\otimes q'$ and $\pi\otimes q_j$ are isotropy-equivalent forms over $F_j$. Invoking Theorem~\ref{km}~$(i)$, we have that $\dim(\pi\otimes q')-i_1(\pi\otimes q')= \dim(\pi\otimes q_j)-i_1(\pi\otimes q_j)$, whereby $$i_1(\pi\otimes q_j)=i_1(\pi\otimes q')+\dim\pi (\dim q_j -\dim q')=i_1(\pi\otimes q')+\dim\pi (i_1(q_j)- 1).$$ Since $i_1(\pi\otimes q')\geqs\dim\pi$ by Corollary~\ref{WScor}, we have that $i_1(\pi\otimes q_j)\geqs (\dim\pi)i_1(q_j)$, whereby the result follows.
\end{proof}





In \cite{Totaro}, Totaro defined a neighbour of a multiple of a Pfister form $\pi$ to be a subform of the multiple of codimension less than $\dim \pi$. The next corollary suggests that this definition can be extended.

\begin{cor}\label{i1cor} Let $q$ a form of dimension at least two and $\pi$ similar to a Pfister form be such that $\pi\otimes q$ is anisotropic over $F$. If $p\subset \pi\otimes q$ over $F$ of codimension less than $(\dim\pi)i_1(q)$, then $p$ is isotropic over $F(\pi\otimes q)$.
\end{cor}



\begin{proof} Theorem~\ref{i1bound} implies that $p$ is a subform of $\pi\otimes q$ of codimension less than $i_1(\pi\otimes q)$, whereby Lemma~\ref{star} implies that $p$ is isotropic over $F(\pi\otimes q)$. 
\end{proof}

We remark that Theorem~\ref{WS} implies that every higher Witt index of $\pi\otimes q$ is a multiple of $\dim\pi$, with the exception of $i_h(\pi\otimes q)$ in the case where $q$ is an odd-dimensional form. At present, we do not have an analogue of Theorem~\ref{i1bound} with respect to $i_r(\pi\otimes q)$ for $2\leqs r\leqs h$. In certain situations, we can establish upper bounds on the values of some higher Witt indices.


\begin{prop}\label{i2} Let $q$ a form of dimension at least two and $\pi$ similar to a Pfister form be such that $\pi\otimes q$ is anisotropic over $F$. Let $F=F_0\subset F_1\subset \ldots \subset F_h$ denote the generic splitting tower of $q$ and $q\simeq q_0, q_1\ldots ,q_h$ the kernel forms of $q$. \begin{enumerate}[$(i)$] 
\item If $\pi\otimes q_1$ is anisotropic over $F_1$, then $i_1(\pi\otimes q)=(\dim\pi)i_1(q)$.
\smallskip
\item Suppose that $q$ is not a Pfister neighbour of codimension at most one and that $\pi\otimes q$ is not similar to a Pfister form, whereby $i_2(q)$ and $i_2(\pi\otimes q)$ are defined. If $\pi\otimes q_1$ is anisotropic over $F_1$ and $\pi\otimes q_2$ is anisotropic over $F_2$, then $i_1(\pi\otimes q)=(\dim\pi)i_1(q)$ and $i_2(\pi\otimes q)\leqs(\dim\pi)i_2(q)$.
\end{enumerate}
\end{prop}

\begin{proof} $(i)$ Considering the extension of $\pi\otimes q$ to the field $F_1$, we have that $$(\pi\otimes q)_{F_1}\simeq \pi_{F_1}\otimes (i_1(q)\x \< 1,-1\>_{F_1} \perp q_1)\simeq ((\dim\pi)i_1(q))\x \< 1,-1\>_{F_1}\perp (\pi_{F_1}\otimes q_1),$$ with $\pi_{F_1}\otimes q_1$ being anisotropic by assumption. As $F(\pi\otimes q)$ is the generic zero field of $\pi\otimes q$, we thus have that $i_1(\pi\otimes q)\leqs (\dim\pi)i_1(q)$. Invoking Theorem~\ref{i1bound}, we also have that $i_1(\pi\otimes q)\geqs (\dim\pi)i_1(q)$, whereby the result follows.

$(ii)$ As forms of height one are necessarily Pfister neighbours of codimension at most one, by \cite[Theorem 5.8]{Kn1} (independently proved by Wadsworth \cite{W}), we have that $q$ and $\pi\otimes q$ are forms of height at least two, whereby $i_2(q)$ and $i_2(\pi\otimes q)$ are defined. Considering the extension of $\pi\otimes q$ to the field $F_2$, we have that $$(\pi\otimes q)_{F_2}\simeq \pi_{F_2}\otimes (\dim\pi (i_1(q)+i_2(q))\x \< 1,-1\>_{F_2} \perp q_2),$$ with $\pi_{F_2}\otimes q_2$ being anisotropic by assumption. As $i_1(\pi\otimes q)=(\dim\pi)i_1(q)$ by $(i)$, we thus have that $i_1(\pi\otimes q)< i((\pi\otimes q)_{F_2})=\dim\pi (i_1(q)+i_2(q))$, whereby $i_2(\pi\otimes q)\leqs(\dim\pi)i_2(q)$.
\end{proof}






With respect to the preceding results, we note the existence of forms $q$ and $\pi$ over $F$ such that the value of $i_1(\pi\otimes q)$ is strictly greater than $(\dim\pi)i_1(q)$. The following example, communicated to me by Karim Becher, can be used to demonstrate this. We will also use this example to show that the converses of Theorem~\ref{R} and \cite[Theorem $3.2$]{F2} do not hold in general.


\begin{example}\label{nocon} Let $q\simeq\< 1,1,1,7\>$ and $\pi\simeq\< 1,1,1,1\>$ over $F=\mathbb{Q}$. Since $\det q\notin \mathbb{Q}^2$, the form $q$ is not similar to a $2$-fold Pfister form. Hence, by invoking the Cassels-Pfister Subform Theorem [L, Ch.X, Theorem 4.5], we may conclude that $q$ is not hyperbolic over $\mathbb{Q}(\pi)$. Moreover, as a consequence of \cite[Corollary $23.4$]{EKM}, it follows that $i_1(q)=1$. Hence, the form $\< 1,1,1\>$ is anisotropic over $\mathbb{Q}(q)$ by Theorem~\ref{km}~$(i)$. As $7\in D_{\mathbb{Q}}(\pi)$, we have that $7\< 1,1,1,1\>\simeq\< 1,1,1,1\>$, and thus that $q\otimes \pi\simeq 16\x\< 1\>$. Hence, we have that $i_1(\pi\otimes q)=8 >(\dim\pi)i_1(q)=4$. Moreover, the Pfister form $q\otimes \pi$ is hyperbolic over $\mathbb{Q}(\pi\otimes\pi)$. Furthermore, as $\< 1,1,1\>\otimes\pi$ is a Pfister neighbour of $16\x\< 1\>$, we have that $\< 1,1,1\>\otimes\pi$ is isotropic over $\mathbb{Q}(q\otimes \pi)$.
\end{example}


As above, the converse to Theorem~\ref{R} does not hold in general. The following result places a necessary condition on situations where this converse holds with respect to all forms over $F$.

\begin{prop}\label{funnycor1} Let $q$ be a form of dimension at least two such that $\pi\otimes q$ is anisotropic over $F$, where $\pi$ is similar to a Pfister form. For all forms $p$ over $F$ such that $\pi\otimes p$ is anisotropic over $F$, suppose that $p$ is isotropic over $F(q)$ whenever $\pi\otimes p$ is isotropic over $F(\pi\otimes q)$. Then $i_1(\pi\otimes q)=\dim\pi(i_1(q))$.
\end{prop}

%
%


\begin{proof} Invoking Theorem~\ref{i1bound}, we have that $i_1(\pi\otimes q)\geqs\dim\pi(i_1(q))$. Suppose, for the sake of contradiction, that $i_1(\pi\otimes q) >\dim\pi(i_1(q))$. Let $p\subset q$ of codimension $i_1(q)$, whereby $p$ is anisotropic over $F(q)$ by Theorem~\ref{km}$(i)$. However, as $\pi\otimes p\subset \pi\otimes q$ of codimension $\dim\pi(i_1(q))$, we have that $\pi\otimes p$ is isotropic over $F(\pi\otimes q)$ by Lemma~\ref{star}, a contradiction. Hence $i_1(\pi\otimes q)=\dim\pi(i_1(q))$.
\end{proof}

Incorporating the above necessary condition, the next result establishes the converse of Theorem~\ref{R} with the aid of one additional assumption.

\begin{prop}\label{funnycor2} Let $p$ and $q$ be forms of dimension at least two such that $\pi\otimes p$ and $\pi\otimes q$ are anisotropic over $F$, where $\pi$ is similar to a Pfister form. Suppose that $i_1(\pi\otimes q)=\dim\pi(i_1(q))$ and that $q$ is isotropic over $F(p)$. If $\pi\otimes p$ is isotropic over $F(\pi\otimes q)$, then $p$ is isotropic over $F(q)$. 
\end{prop}

\begin{proof} Suppose that $\pi\otimes p$ is isotropic over $F(\pi\otimes q)$. Invoking Theorem~\ref{km}$(i)$, we have that $\dim (\pi\otimes p)-i_1(\pi\otimes p)\geqs \dim\pi(\dim q-i_1(q))$, whereby it follows that $i_1(\pi\otimes p)\leqs \dim\pi(\dim p-\dim q+i_1(q))$. Invoking Theorem~\ref{i1bound}, we have that $\dim\pi (i_1(p))\leqs \dim\pi(\dim p-\dim q+i_1(q))$. As $q$ is isotropic over $F(p)$ by assumption, we have that $\dim q-i_1(q)\geqs \dim p-i_1(p)$ by Theorem~\ref{km}$(i)$. Hence, $\dim q-i_1(q)= \dim p-i_1(p)$ , whereby $p$ is isotropic over $F(q)$ by Theorem~\ref{km}$(ii)$.
\end{proof}

%

The preceding results provide some additional motivation for determining when the equality $i_1(\pi\otimes q)=\dim\pi(i_1(q))$ holds, a question which naturally arises in light of Theorem~\ref{i1bound}. Our next results establish conditions on the form $q$ that ensure that this equality holds. Over real fields, we can establish the following statement.

\begin{prop}\label{pivalues} Let $q$ a form of dimension at least two and $\pi$ similar to a Pfister form be such that $\pi\otimes q$ is anisotropic over a real field $F$. Let $P$ be an ordering of $F$ such that $\pi$ is definite at $P$ and $q$ is indefinite at $P$, whereby $\abs{\sgn_P(q)} \leqs \dim q-2i_1(q)$. If $\abs{\sgn_P(q)} = \dim q-2i_1(q)$, then $i_1(\pi\otimes q)= (\dim\pi)i_1(q)$.
\end{prop}

\begin{proof} As $q$ is indefinite at $P$, Theorem~\ref{ELW} implies that $P$ extends to $F(q)$, whereby it follows that $\abs{\sgn_P(q)} \leqs \dim q-2i_1(q)$. By assumption, we have that $\pi$ is (positive) definite at $P$ and $\abs{\sgn_P(q)} = \dim q-2i_1(q)$, whereby it follows that $$\abs{\sgn_P(\pi\ox q)} =\dim\pi( \dim q-2i_1(q)).$$ Hence, Theorem~\ref{ELW} implies that $P$ extends to $K=F(\pi\ox q)$. Over $K$, $(\pi\ox q)_K\simeq ((\pi\ox q)_{K})_{\textrm{an}}\perp i_1(\pi\ox q)\x\< 1,-1\>_K$, whereby a comparison of signatures with respect to $P$ yields that $i_1(\pi\ox q)\leqs (\dim\pi)i_1(q)$. Invoking Theorem~\ref{i1bound}, we also have that $i_1(\pi\ox q)\geqs (\dim\pi)i_1(q)$, whereby the result follows.
\end{proof}


In the preceding result, we established that the value of $i_1(\pi\otimes q)$ coincides with our lower bound when $i_1(q)$ is maximal with respect to the signature of $q$ at an ordering $P$. In a similar spirit, if $i_1(q)$ is maximal with respect to the dimension of $q$, we can also establish this equality.




\begin{prop}\label{i1maxspl} Let $q$ a form of dimension at least two and $\pi$ similar to a Pfister form be such that $\pi\otimes q$ is anisotropic over $F$. If $q$ has maximal splitting, then $i_1(\pi\otimes q)= (\dim\pi)i_1(q)$.
\end{prop}

\begin{proof} Let $\dim q= 2^n+k$ for some integers $n$ and $k$ such that $0< k\leqs 2^n$. Hence, $\dim(\pi\otimes q)=2^n\dim\pi+k\dim\pi$, where $0< k\dim\pi\leqs 2^n\dim\pi$. As $i_1(q)=k$, 
Theorem~\ref{i1bound} implies that $i_1(\pi\otimes q)\geqs k\dim\pi$. Let $\vartheta\subset\pi\otimes q$ over $F$ such that $\dim\vartheta =2^n\dim\pi$. If $i_1(\pi\otimes q)> k\dim\pi$, then Lemma~\ref{star} implies that $\vartheta$ is isotropic over $F(\pi\otimes q)$, contradicting Theorem~\ref{H95}. Thus, $i_1(\pi\otimes q)= (\dim\pi)i_1(q)$.
\end{proof}


%

\begin{cor}\label{maxsplcor} Let $q$ a form of dimension at least two and $\pi$ similar to a Pfister form be such that $\pi\otimes q$ is anisotropic over $F$. If $q$ has maximal splitting, then $\pi\otimes q$ has maximal splitting.
\end{cor}

\begin{proof} Proposition~\ref{i1maxspl} implies that $\dim(\pi\otimes q)-i_1(\pi\otimes q)=\dim\pi(\dim q-i_1(q))$. Since $\dim q-i_1(q)=2^k$ for some integer $k\geqs 0$, it follows that $\dim(\pi\otimes q)-i_1(\pi\otimes q)=2^{n+k}$ for some $n\in\mathbb{N}$. Hence, $\pi\otimes q$ has maximal splitting.
\end{proof}



Corollary~\ref{maxsplcor} was previously known to hold in the case where $\dim q=2^n+1$ for some $n\in\mathbb{N}$, where the statement follows through combining Theorem~\ref{H95} with Corollary~\ref{WScor} (see \cite[Corollary 8.9]{I}). 

Letting $\tau$ be a neighbour of a Pfister form $\pi$, we note that the statements of Proposition~\ref{i1maxspl} and Corollary~\ref{maxsplcor} hold with respect to the product $\tau\otimes q$ in the case where the codimension of $\tau\otimes q$ as a subform of $\pi\otimes q$ is less than $i_1(\pi\otimes q)$. As it can be difficult to determine the exact value of $i_1(\pi\otimes q)$ for a prescribed form $q$, we will invoke Theorem~\ref{i1bound} to express this observation in terms of $i_1(q)$.



\begin{cor}\label{PNmaxsplcor} Let $q$ a form of dimension at least two and $\pi$ similar to a Pfister form be such that $\pi\otimes q$ is anisotropic over $F$. Let $\tau$ be a neighbour of $\pi$ such that $\dim\tau>\dim\pi-\frac{(\dim\pi)i_1(q)}{\dim q}$. If $q$ has maximal splitting, then $i_1(\tau\otimes q)=(\dim\pi)i_1(q)-(\dim q)(\dim\pi-\dim \tau)$, whereby $\tau\otimes q$ has maximal splitting.
\end{cor}



\begin{proof} As per Corollary~\ref{i1cor}, since $\dim(\tau\otimes q)> \dim(\pi\otimes q)-(\dim\pi)i_1(q)$, it follows that $\tau\otimes q$ is isotropic over $F(\pi\otimes q)$. Thus, $\tau\otimes q$ is isotropy equivalent to $\pi\otimes q$, whereby Theorem~\ref{km}~$(ii)$ implies that $i_1(\tau\otimes q)=i_1(\pi\otimes q)-(\dim q)(\dim\pi-\dim \tau)$. Invoking Corollary~\ref{maxsplcor}, it follows that $i_1(\tau\otimes q)=(\dim\pi)i_1(q)-(\dim q)(\dim\pi-\dim \tau)$, whereby $\dim(\tau\otimes q)-i_1(\tau\otimes q)=\dim\pi(\dim q-i_1(q))=2^m$ for some $m\in\mathbb{N}$, whereby $\tau\otimes q$ has maximal splitting.\end{proof}


The following example shows that the dimension condition in the preceding result can be sharp. This example furthermore demonstrates that the anisotropic product of two Pfister neighbours, both necessarily having maximal splitting, need not have maximal splitting.



%
%



\begin{example}\label{2pns} Let $F$ be a field such that $\< 1,1,1,d,d,d\>$ is anisotropic over $F$ for some $d\in F^{\x}$. Let $K=F(\!(x)\!)(\!(y)\!)$ be the iterated Laurent series field in two variables over $F$. Consider the Pfister neighbours $\tau_1\simeq\<1,1,1\>$ and $\tau_2\simeq\< d\>\perp\< 1,-x,-y,xy\>$ over $K$. Applying Springer's Theorem for complete discretely valued fields \cite[Theorem VI.$1.4$]{LAM}, once with respect to the $y$-adic valuation and subsequently twice with respect to the $x$-adic valuation, one sees that the form $\tau_1\otimes\tau_2$ is anisotropic over $K$. Suppose $\tau_1\otimes\tau_2$ has maximal splitting. Since $\dim (\tau_1\otimes\tau_2)=15$, \cite[Corollary 3]{H} implies that $\tau_1\otimes\tau_2$ is a neighbour of some $4$-fold Pfister form $\pi$ over $K$. Comparing determinants, we have that $\tau_1\otimes\tau_2\perp\< d\>\simeq a\pi$ for some $a\in K$. As $a\pi\in I^3K$, it has trivial Clifford invariant (see \cite[Corollary V.$3.4$]{LAM}), whereby the Clifford invariant of $\tau_1\otimes\tau_2\perp\< d\>$ must also be trivial. Hence, applying \cite[V.$(3.13)$]{LAM}, we obtain that $C_0(\tau_1\otimes\tau_2)$, the even Clifford algebra of $\tau_1\otimes\tau_2$, belongs to the trivial class in the Brauer group over $K$. However, applying \cite[V.$(3.13)$]{LAM} to a decomposition of $\tau_1\otimes\tau_2$, we see that $C_0(\tau_1\otimes\tau_2)$ is Brauer equivalent to $\openquat {-1} {-1} K\otimes_K \openquat {x} {y} K$, a product of two quaternion algebras. As above, iterated applications of Springer's Theorem \cite[Theorem VI.$1.4$]{LAM} with respect to the $y$-adic and $x$-adic valuations enable us to conclude that the form $\< 1,1,1,x,y,-xy\>$ is anisotropic over $K$, whereby \cite[Theorem III.$4.8$]{LAM} implies that $\openquat {-1} {-1} K\otimes_K \openquat {x} {y} K$ is a biquaternion division algebra over $K$, and hence non-trivial in the Brauer group of $K$. Thus, we may conclude that $\tau_1\otimes\tau_2$ does not have maximal splitting.
\end{example}

\section{Multiples of generic Pfister forms}



A number of form-theoretic properties are preserved under multiplication by an arbitrary Pfister form. This is clearly the case with respect to the properties of being a neighbour or a product of a Pfister form. As discussed previously, the properties of being excellent, round or, as established in Corollary~\ref{maxsplcor}, having maximal splitting are also known to be preserved. While many properties of a form are reflected in the properties of its Pfister multiples, we cannot expect the converse to hold in general. Despite this, it seems reasonable to suggest that the behaviour of the generic Pfister multiples of a form might mirror the behaviour of the form itself, with \cite[Lemma 5.4]{Ilow} and \cite[Proposition 6.4.3 and Corollaire 6.4.7]{R} appearing to support this view.

Thus, given a form $q$ over a field $F$, we will consider what can be said regarding its relationship with the form $\pi\otimes q$, where $\pi$ is the generic Pfister form $\<\!\< -x_1,\ldots ,-x_n\>\!\>\simeq \< 1,x_1\>\otimes\ldots\otimes\< 1,x_n\>$ defined over the iterated Laurent series field $F(\!(x_1)\!)\ldots (\!(x_n)\!)$ (we remark that these results also hold if $\pi$ is considered as a form over the rational function field $F(x_1,\ldots ,x_n)$).

\begin{prop}\label{i1trans} Let $q$ be an anisotropic form over $F$ of dimension at least two. For $\pi\simeq \<\!\< -x_1,\ldots ,-x_n\>\!\>$, the anisotropic form $q\otimes \pi$ over $F(\!(x_1)\!)\ldots (\!(x_n)\!)$ satisfies $i_1(q\otimes\pi)=(\dim\pi)i_1(q)$.
\end{prop}


\begin{proof} Consider the form $q\otimes \< 1,x_1\>$ over $F(\!(x_1)\!)$, which is anisotropic by Lemma~\ref{Hlemma}. Invoking Theorem~\ref{i1bound}, we have that $i_1(q\otimes \< 1,x_1\>)\geqs 2i_1(q)$. Suppose, for the sake of contradiction, that $i_1(q\otimes \< 1,x_1\>) > 2i_1(q)$. Hence, since the form $q\otimes \< 1,x_1\>$ is isotropic over $F(q)(\!(x_1)\!)$, it follows that $i\left({\left(q\otimes \< 1,x_1\>\right)}_{F(q)(\!(x_1)\!)}\right)> 2i_1(q)$. However, Lemma~\ref{Hlemma} implies that $i\left({q\otimes \< 1,x_1\>}_{F(q)(\!(x_1)\!)}\right)=i(q_{F(q)})+i(q_{F(q)})=2i_1(q)$, in contradiction to the above. Thus, we have that $i_1(q\otimes \< 1,x_1\>)= 2i_1(q)$. The result follows by iterating the above argument.
\end{proof}

%



%
%
%
%
%


\begin{question}\label{spltransquestion} Let $q$ be an anisotropic form over $F$ of dimension at least two with splitting pattern $(i_1(q),\ldots ,i_h(q))$. For $\pi\simeq \<\!\< -x_1,\ldots ,-x_n\>\!\>$, is the splitting pattern of $q\otimes \pi$ over $F(\!(x_1)\!)\ldots (\!(x_n)\!)$ given by $((\dim\pi)i_1(q),\ldots ,(\dim\pi)i_h(q))$?
\end{question}



\begin{prop}\label{maxtrans} Let $q$ be an anisotropic form over $F$ of dimension at least two. For $\pi\simeq \<\!\< -x_1,\ldots ,-x_n\>\!\>$, the form $q\otimes \pi$ over $F(\!(x_1)\!)\ldots (\!(x_n)\!)$ has maximal splitting if and only if $q$ has maximal splitting.
\end{prop}


\begin{proof} Assuming that $q$ has maximal splitting, Corollary~\ref{maxsplcor} implies that $q\otimes \pi$ has maximal splitting.

Conversely, let us assume that $q\otimes \pi$ has maximal splitting. Letting $\dim q=2^m+k$, where $0< k\leqs 2^m$, we have that $\dim (q\otimes \pi )=2^{m+n}+k(2^n)$. As $0< k(2^n)\leqs 2^{m+n}$, we have that $i_1( q\otimes \pi )=k(2^n)$. Invoking Proposition~\ref{i1trans}, it follows that $i_1(q)=k$, whereby $q$ has maximal splitting.
\end{proof}


%


%
%


\begin{prop}\label{pntrans} Let $q$ be an anisotropic form over $F$ and let $\rho$ be an anisotropic Pfister form over $F$. Let $\pi\simeq \<\!\< -x_1,\ldots ,-x_n\>\!\>$ over $K=F(\!(x_1)\!)\ldots (\!(x_n)\!)$.\begin{enumerate}[$(i)$]
\item $q\otimes \pi$ is a neighbour of a Pfister form over $K$ if and only if $q\otimes \pi$ is a neighbour of $\sigma\otimes\pi$ for $\sigma$ some Pfister form over $F$.
\smallskip
\item $q\otimes \pi$ is a neighbour of the Pfister form $\rho\otimes\pi$ over $K$  if and only if $q$ is a neighbour of the Pfister form $\rho$ over $F$.
\end{enumerate}
\end{prop}

\begin{proof} $(i)$ The ``if'' statement is clear. To prove the converse, we begin by considering the form $q\otimes \< 1,x_1\>$, which is anisotropic over $F(\!(x_1)\!)$ by Lemma~\ref{Hlemma}. Suppose that $q\otimes \< 1,x_1\>$ is a neighbour of an anisotropic Pfister form $\gamma$ over $F(\!(x_1)\!)$. As the form $q\otimes \< 1,x_1\>$ is isotropic (indeed, hyperbolic) over $F(\!(x_1)\!)(\< 1,x_1\>)$, it follows that the Pfister form $\gamma$ is hyperbolic over $F(\!(x_1)\!)(\< 1,x_1\>)$, whereby we have that $\gamma\simeq \< 1,x_1\>\otimes \vartheta$ for some Pfister form $\vartheta$ over $F(\!(x_1)\!)$ (see \cite[Ch.X, Theorem $4.11$ and Corollary $4.13$]{LAM}). As was observed in the proof of \cite[Proposition 7]{H}, since every non-zero square class in $F(\!(x_1)\!)$ can be represented by $a$ or $ax_1$ for some $a\in F^{\x}$, and since $\<\!\< -ax_1,-bx_1\>\!\>\simeq \<\!\< -ab,-ax_1\>\!\>$ for all $a,b\in F^{\x}$, we may assume that either $\vartheta$ is defined over $F$ or that $\vartheta\simeq\delta\otimes \< 1,ax_1\>$ for $a\in F^{\x}$ and $\delta$ a Pfister form over $F$. In the latter case, we have that $$\gamma\simeq \< 1,x_1\>\otimes \vartheta\simeq  \< 1,x_1\>\otimes \delta\otimes \< 1,ax_1\>\simeq \< 1,x_1\>\otimes \delta\otimes \< 1,a\>.$$ Hence, in either case, we have that $\gamma\simeq \< 1,x_1\>\otimes \sigma$ for $\sigma$ a Pfister form over $F$. Statement $(i)$ now follows by iterating the above argument.

$(ii)$ The ``if'' statement is clear. To prove the converse, we begin by assuming that $q\otimes \< 1,x_1\>$ is a neighbour of $\rho\otimes \< 1,x_1\>$ over $F(\!(x_1)\!)$ for some Pfister form $\rho$ over $F$. Letting $\gamma$ denote the complementary form of $q\otimes \< 1,x_1\>$, we have that $q\otimes \< 1,x_1\>\perp\gamma\sim \rho\otimes \< 1,x_1\>$. As every square class in $F(\!(x_1)\!)$ is represented by $a$ or $ax_1$ for some $a\in F^{\times}$, and since $x_1\in D(\< 1,x_1\>)=G(\< 1,x_1\>)$, we have that $q\otimes \< 1,x_1\>\perp\gamma\simeq a(\rho\otimes \< 1,x_1\>)$ for $a\in F^{\times}$. As $\gamma\simeq \gamma_1\perp x_1\gamma_2$ for some forms $\gamma_1$, $\gamma_2$ over $F$, we have that $(q\perp x_1 q)\perp (\gamma_1\perp x_1\gamma_2)\simeq a\rho\perp x_1(a\rho).$ Invoking Lemma~\ref{Hlemma}, it thus follows that $$q\perp\gamma_1\simeq a\rho\simeq q\perp\gamma_2.$$ Thus, $q$ is a Pfister neighbour of $\tau$ with complementary form $\gamma_1\simeq\gamma_2$. Statement $(ii)$ now follows by iterating the above argument.
\end{proof}


\begin{remark}\label{Hcontext} In \cite[Proposition 7]{H} it was established that a form $q$ over $F$ is a Pfister neighbour if and only if it is a Pfister neighbour over $F(x)$. Thus, Proposition~\ref{pntrans} may be viewed as a generalisation of this result in the case where $q$ is an anisotropic form. Moreover, by adapting the isotropic part of the proof of \cite[Proposition 7]{H} and invoking Proposition~\ref{pntrans} in the anisotropic case, one can establish that, for all forms $q$ over $F$, we have that $q\otimes \<\!\< -x_1,\ldots ,-x_n\>\!\>$ is a Pfister neighbour over $F(\!(x_1)\!)\ldots (\!(x_n)\!)$ if and only if $q$ is a Pfister neighbour over $F$.
\end{remark}

As was remarked in \cite{Kn2}, if $q$ is an excellent form and $\pi$ is a Pfister form, then $q\otimes \pi$ is an excellent form.

\begin{prop}\label{extrans} Let $q$ be an anisotropic form over $F$. For $\pi\simeq \<\!\< -x_1,\ldots ,-x_n\>\!\>$, the form $q\otimes \pi$ over $F(\!(x_1)\!)\ldots (\!(x_n)\!)$ is excellent if and only if $q$ is excellent.
\end{prop}


\begin{proof} As above, it suffices to prove the ``only if'' statement. We will begin by assuming that $q\otimes \< 1,x_1\>$ is excellent over $F(\!(x_1)\!)$. Hence, $q\otimes \< 1,x_1\>$ is a Pfister neighbour over $F(\!(x_1)\!)$. Invoking Proposition~\ref{pntrans} $(i)$, we have that $q\otimes \< 1,x_1\>$ is a Pfister neighbour of $\rho\otimes\< 1,x_1\>$ for some Pfister form $\rho$ over $F$. As per the proof of Proposition~\ref{pntrans} $(ii)$, $q\otimes \< 1,x_1\>$ has complementary form $q_1\otimes\< 1,x_1\>$ for some form $q_1$ over $F$. Moreover, again as per the proof of Proposition~\ref{pntrans} $(ii)$, we have that $q$ is a neighbour of $\rho$ with complementary form $q_1$. As $q\otimes \< 1,x_1\>$ is excellent, we have that $q_1\otimes\< 1,x_1\>$ is excellent. Arguing as above, we may establish that $q_1\otimes\< 1,x_1\>$ has complementary form $q_2\otimes\< 1,x_1\>$ for some form $q_2$ over $F$, and that $q_1$ is a Pfister neighbour with complementary form $q_2$. Iterating this argument until $\dim q_n\leqs 1$, we thus obtain that $q$ is an excellent form over $F$. 

The general statement now follows by iterating the above argument.
\end{proof}



%


\begin{prop}\label{divtrans} Let $q$ be an anisotropic form over $F$ and let $\rho$ be an anisotropic $m$-fold Pfister form over $F$. Let $\pi\simeq \<\!\< -x_1,\ldots ,-x_n\>\!\>$ over $K=F(\!(x_1)\!)\ldots (\!(x_n)\!)$. 
\begin{enumerate}[$(i)$]
\item $q\otimes \pi$ is a multiple of $\vartheta\in P_{m+n}K$ if and only if $q\otimes \pi$ is a multiple of $\sigma\otimes \pi$ for some $\sigma\in P_mF$.
\smallskip
\item $q\otimes \pi$ is a multiple of $\rho\otimes \pi$ over $K$ if and only if $q$ is a multiple of $\rho$ over $F$.
\end{enumerate}
\end{prop}

\begin{proof} $(i)$ The ``if'' statement is clear. To prove the converse, we begin by considering the case where $q\otimes \< 1,x_1\>$ is a multiple of $\vartheta\in P_{m+1}F(\!(x_1)\!)$. As per the proof of Proposition~\ref{pntrans} $(i)$, we may assume that either $\vartheta\simeq  \<\!\< -a_1,\ldots ,-a_{m+1}\>\!\>$ or $\vartheta\simeq  \<\!\< -a_1,\ldots ,-a_m,-a_{m+1}x_1\>\!\>$ for some $a_1,\ldots ,a_{m+1}\in F^{\x}$.


In the case where $\vartheta\simeq  \<\!\< -a_1,\ldots ,-a_{m+1}\>\!\>$, let $\varphi$ be a form over $F(\!(x_1)\!)$ such that $q\otimes \< 1,x_1\>\simeq \vartheta\otimes\varphi$. As $\varphi\simeq \varphi_1\perp x_1\varphi_2$ for some forms $\varphi_1$ and $\varphi_2$ over $F$, we have that $q\otimes \< 1,x_1\>\simeq \vartheta\otimes\varphi_1\perp x_1(\vartheta\otimes\varphi_2)$, whereby $\vartheta\otimes\varphi_1\simeq q\simeq \vartheta\otimes\varphi_2$. Hence $$q\otimes \< 1,x_1\>\simeq (\vartheta\otimes\varphi_1)\otimes \< 1,x_1\>\simeq (\vartheta\otimes \< 1,x_1\>)\otimes\varphi_1,$$ whereby $q\otimes \< 1,x_1\>$ is a multiple of $\vartheta\otimes \< 1,x_1\>$ for $\vartheta\in P_{m+1}F$ in this case.

In the case where $\vartheta\simeq \<\!\< -a_1,\ldots ,-a_m,-a_{m+1}x_1\>\!\>$, we will denote $\<\!\< -a_1,\ldots ,-a_m\>\!\>$ by $\sigma$, whereby $\vartheta\simeq \sigma\otimes \< 1,a_{m+1}x_1\>$. Let $\varphi$ be a form over $F(\!(x_1)\!)$ such that $q\otimes \< 1,x_1\>\simeq \vartheta\otimes\varphi$. As $\varphi\simeq \varphi_1\perp x_1\varphi_2$ for some forms $\varphi_1$ and $\varphi_2$ over $F$, we have that 
\begin{align*} 
q\otimes \< 1,x_1\>&\simeq \vartheta\otimes\varphi\simeq (\sigma\otimes\< 1,a_{m+1}x_1\>)\otimes (\varphi_1\perp x_1\varphi_2),\\ 
&\simeq \sigma\otimes\varphi_1\perp a_{m+1}(\sigma\otimes\varphi_2)\perp x_1(a_{m+1}(\sigma\otimes\varphi_1)\perp \sigma\otimes\varphi_2),\\
&\simeq \sigma\otimes (\varphi_1\perp a_{m+1}\varphi_2)\perp x_1(\sigma\otimes (\varphi_2\perp a_{m+1}\varphi_1)).
\end{align*}
Hence, by taking the difference of isometric forms and invoking Lemma~\ref{Hlemma}, it follows that $$\sigma\otimes (\varphi_1\perp a_{m+1}\varphi_2)\simeq q\simeq \sigma\otimes (\varphi_2\perp a_{m+1}\varphi_1).$$ Thus, we have that $$q\otimes \< 1,x_1\>\simeq (\sigma\otimes (\varphi_1\perp a_{m+1}\varphi_2))\otimes \< 1,x_1\>\simeq (\varphi_1\perp a_{m+1}\varphi_2)\otimes (\sigma\otimes \< 1,x_1\>),$$ whereby $q\otimes \< 1,x_1\>$ is a multiple of $\sigma\otimes \< 1,x_1\>$ for $\sigma\in P_{m}F$ in this case.

 Statement $(i)$ follows by iterating the above argument.

$(ii)$ The ``if'' statement is clear. To prove the converse, we begin by considering the case where $q\otimes \< 1,x_1\>$ is a multiple of $\rho\otimes \< 1,x_1\>$ over $F(\!(x_1)\!)$. Let $\varphi$ be a form over $F(\!(x_1)\!)$ such that $q\otimes \< 1,x_1\>\simeq \varphi \otimes(\rho\otimes \< 1,x_1\>)$. As $\varphi\simeq \varphi_1\perp x_1\varphi_2$ for some forms $\varphi_1$ and $\varphi_2$ over $F$, we have that $$q\otimes \< 1,x_1\>\simeq (\varphi_1\otimes \rho\perp \varphi_2\otimes\rho)\perp x_1(\varphi_1\otimes \rho\perp \varphi_2\otimes\rho).$$ Hence, the difference of these forms is hyperbolic, whereby we may invoke Lemma~\ref{Hlemma} to establish that $q\simeq\varphi_1\otimes \rho\perp \varphi_2\otimes\rho\simeq (\varphi_1\perp\varphi_2)\otimes\rho$, as desired. Statement $(ii)$ now follows by iterating the above argument.
\end{proof}

Witt's Round Form Theorem \cite[Theorem X.$1.14$]{LAM} states that the product of a Pfister form and a round form is round.

\begin{prop}\label{roundtrans} Let $q$ be an anisotropic form over $F$. For $\pi\simeq \<\!\< -x_1,\ldots ,-x_n\>\!\>$, the form $q\otimes \pi$ over $F(\!(x_1)\!)\ldots (\!(x_n)\!)$ is round if and only if $q$ is round.
\end{prop}


\begin{proof} As above, it suffices to prove the ``only if'' statement. We will begin by assuming that $q\otimes \< 1,x_1\>$ is round over $F(\!(x_1)\!)$. 

Hence, $1\in D_{F(\!(x_1)\!)}(q\otimes \< 1,x_1\>)=G_{F(\!(x_1)\!)}(q\otimes \< 1,x_1\>)$, whereby $\< -1\>\perp q\perp x_1q$ is isotropic over $F(\!(x_1)\!)$. Invoking Lemma~\ref{Hlemma}, we obtain that $\< -1\>\perp q$ is isotropic over $F$, whereby $1\in D_F(q)$ and thus $G_F(q)\subset D_F(q)$.


Let $y\in D_F(q)$. As $y\in D_{F(\!(x_1)\!)}(q\perp x_1q)=G_{F(\!(x_1)\!)}(q\perp x_1q)$, it follows that $q\perp x_1q\simeq yq\perp x_1yq$ over $F(\!(x_1)\!)$. Thus, $q\perp -yq\perp x_1(q\perp -yq)$ is hyperbolic over $F(\!(x_1)\!)$. Invoking Lemma~\ref{Hlemma}, it follows that $q\perp -yq$ is hyperbolic over $F$. Hence, we have that $y\in G_F(q)$, whereby $G_F(q)=D_F(q)$.

The general statement now follows by iterating the above argument.
\end{proof}


%

%
%



\section{Properties preserved by Pfister products}

%



As before, many properties of a form are preserved under multiplication by a Pfister form. Moreover, as per the results of the previous section, the absence of these properties is reflected in the generic Pfister multiples of the form. Thus, combining these observations, one can look to clarify how such properties relate to one another.


%




%
%
%
%


%


An important question in this regard, first posed in \cite{H}, is to determine when the maximal splitting property implies that the Pfister neighbour property also holds. In particular, a condition that refers solely to the dimension of the form is sought. Characterisations of the Pfister neighbour property are known for forms of small dimension (see \cite{Kn2} for example). In particular, an anisotropic $5$-dimensional form is a Pfister neighbour if and only if its even Clifford algebra is of Schur index two, as follows from \cite[Corollary 8.2]{Kn2}. Thus, over the field $F=\mathbb{R}(w,x,y,z)$, the form $q\simeq \< 1,w,x,y,z\>$ is not a Pfister neighbour: by applying \cite[V.$(3.13)$]{LAM}, one sees that its even Clifford algebra is Brauer equivalent to $(-w,-x)_F\otimes (-yz,wxz)_F$, which is a biquaternion division algebra by Albert's Theorem \cite[Theorem III.$4.8$]{LAM}. Moreover, $q$ trivially has maximal splitting. For all $n> 2$, Hoffmann considered the product $\pi_{n-2}\otimes q$, where $\pi_{n-2}$ is the $(n-2)$-fold Pfister form $\<\!\< -1,\ldots ,-1\>\!\>$ over $F$. In \cite[Example 2]{H}, he established that $\pi_{n-2}\otimes q$ is a $(2^n+2^{n-2})$-dimensional form with maximal splitting that is not a Pfister neighbour. The following proposition allows us to recover the existence of such $(2^n+2^{n-2})$-dimensional forms. More generally, given any form $q$ with maximal splitting that is not a Pfister neighbour, for all $n\in\mathbb{N}$ there exists an $n$-fold Pfister multiple of $q$ with maximal splitting that is not a Pfister neighbour.



\begin{prop}\label{lowerbound} Let $q$ be an anisotropic form over a field $F$ that has maximal splitting but is not a Pfister neighbour. For $n\in\mathbb{N}$ and $\pi\simeq \<\!\< -x_1,\ldots ,-x_n\>\!\>$ over $F(\!(x_1)\!)\ldots (\!(x_n)\!)$, the form $\pi\otimes q$ has maximal splitting but is not a Pfister neighbour.
\end{prop}

\begin{proof} We note that $\pi\otimes q$ has maximal splitting by Corollary~\ref{maxsplcor}. The fact that $\pi\otimes q$ is not a Pfister neighbour follows from Proposition~\ref{pntrans}.
\end{proof}


For $F$, $q$ and $\pi$ as in Proposition~\ref{lowerbound}, let $m\in\mathbb{N}$ be such that $2^m < 2^n(\dim q)\leqs 2^{m+1}$. For $d\in\mathbb{N}$ such that $2^m< d\leqs 2^n(\dim q)$, every $d$-dimensional subform of $\pi\otimes q$ has maximal splitting but is not a Pfister neighbour. In particular, as observed in \cite[Proposition 1.5]{IV}, for $n \geqs 2$ and $d\in\mathbb{N}$ such that $2^n< d\leqs 2^{n}+2^{n-2}$, there exists a $d$-dimensional form with maximal splitting that is not a Pfister neighbour. The following conjecture of Izhboldin and Vishik posits that the dimensions of such forms necessarily belong to these intervals.

%


{\textbf{[12, Conjecture 1.6].}} Let $F$ be a field and $q$ be an anisotropic form over $F$ with maximal splitting. If $2^{n}+2^{n-2} < \dim q\leqs 2^{n+1}$ holds for $n \geqs 2$, then $q$ is a Pfister neighbour.

%
%



\cite[Conjecture 1.6]{IV} appears to be very difficult to resolve. It is known to hold for $n\leqs 3$ (see \cite{Hspl} or \cite{Imax}). In order to establish the truth of this conjecture for a fixed value of $n\geqs 4$, we remark that it suffices to prove the statement in the (hardest) case where $q$ is any form of dimension $2^n+2^{n-2}+1$. More generally, one can look to prove \cite[Conjecture 1.6]{IV} with respect to forms $q$ of some prescribed dimension, an approach which has been successfully employed by a number of authors. For $n\geqs 4$, the conjecture is known to hold when $2^{n+1}-7\leqs \dim q\leqs 2^{n+1}$ (see \cite[Theorem 1.7]{IV}).

%
%
%
%
%

We remark that Proposition~\ref{lowerbound} is also of some relevance to these approaches towards resolving \cite[Conjecture 1.6]{IV}. In particular, in order to establish the conjecture with respect to a form $q$, Proposition~\ref{lowerbound} implies that it is sufficient to prove the statement with respect to an $m$-fold generic Pfister multiple of $q$ for any prescribed $m\in\mathbb{N}$. Thus, it suffices to prove the conjecture with respect to the forms belonging to any prescribed power of the fundamental ideal (generated by even-dimensional forms). Hence, when treating the general conjecture, there is no loss of generality in assuming that the first $m$ cohomological invariants of $q$ are trivial. The same considerations apply when seeking to establish the conjecture with respect to forms of prescribed dimension. Thus, to prove the conjecture with respect to $24$-dimensional forms for example, it suffices to prove the statement with respect to $48$-dimensional forms with trivial discriminant (although this is unlikely to be easier).




We conclude our considerations of \cite[Conjecture 1.6]{IV} by invoking descent results of Laghribi to establish the conjecture with respect to forms with specified properties.



\begin{prop}\label{implicit} For $n\geqs 4$, let $q$ be an anisotropic form over $F$ of dimension at least $2^{n+1}-10$. Suppose that $q$ contains a subform $p$ of one of the following types: \begin{enumerate}[$(i)$] \item $\dim p=2^{n+1}-10$, $\det p=-1$ and the Clifford algebra of $p$ has Schur index at most two, \item $\dim p=2^{n+1}-9$ and the even Clifford algebra of $p$ has Schur index at most two, \item $\dim p=2^{n+1}-8$ and the Clifford algebra of $p$ extended to $F(p)(\sqrt{\det p})$ has Schur index at most two.\end{enumerate} If $q$ has maximal splitting, then $q$ is a Pfister neighbour.
\end{prop}

%
%

\begin{proof} Assuming that $q$ has maximal splitting, we have that $q$ and $p$ are isotropy equivalent by Lemma~\ref{star}, whereby $p$ has maximal splitting by Theorem~\ref{km}$(ii)$. We consider the extension of $p$ to $F(p)$. Invoking \cite[Th\'eor\`eme principal]{Lag}, we have that $(p_{F(p)})_\text{an}$ is defined over $F$. Hence, we have that $p$ is a Pfister neighbour, by \cite[Theorem $7.13$]{Kn2}, whereby it follows that $q$ is a Pfister neighbour.
\end{proof}

As per Proposition~\ref{i1trans}, the value of the first Witt index of a form is reflected in that of its generic Pfister multiples. In \cite{Hspl}, a complete classification of the splitting patterns of forms of dimension at most $9$ was given. For $q$ an even-dimensional form of dimension at least four and at most eight, it is known that $i_1(q)$ is divisible by two if and only if $q$ is a multiple of a $1$-fold Pfister form. In \cite[Proof of Conjecture $0.10$]{I}, Izhboldin proved that a $10$-dimensional form $q$ satisfies $i_1(q)=2$ if and only if $q$ is a multiple of a $1$-fold Pfister form or $q$ is a Pfister neighbour. As per \cite{Hspl} or \cite{Imax}, a $12$-dimensional form $q$ satisfies $i_1(q)=4$ if and only if $q$ is a Pfister neighbour, in which case $q$ is a multiple of a $2$-fold Pfister form. As per \cite[pp. $94$-$95$]{IKKV}, Vishik established that a $12$-dimensional form $q$ satisfies $i_1(q)=2$ if and only if its splitting pattern is of the form $(2,4)$ (in which case it is a multiple of a $2$-fold Pfister form) or of the form $(2,2,2)$ (with Vishik hypothesising that $q$ is a multiple of a $1$-fold Pfister form in this case). Totaro classified $14$-dimensional forms with first Witt index greater than one in \cite[Theorem 4.2]{Totaro}, determining that such a form $q$ satisfies $i_1(q)=2$ if and only if $q$ is a multiple of a $1$-fold Pfister form or $q$ is a subform of a $16$-dimensional multiple of a $2$-fold Pfister form. 

Thus, assuming Vishik's hypothesis is true, we have that an even-dimensional form $q$ of dimension less than $16$ satisfies $i_1(q)=2$ if and only if $q$ is isotropy equivalent to a multiple of a Pfister form. Vishik showed that this phenomenon does not extend further by constructing a $16$-dimensional form $q$ satisfying $i_1(q)=2$ that is not a multiple of a $2$-fold Pfister form. To our knowledge, Vishik's example is the first example of a form having non-trivial first Witt index that is not isotropy equivalent to a multiple of a Pfister form. Vishik's form, having splitting pattern $(2,2,2,2)$, is the first example of a form whose higher Witt indices are all even but is not a multiple of a Pfister form (the converse holds as a consequence of Theorem~\ref{WS}).



\begin{example}\label{vishikform} As presented in \cite[Lemma 7.1]{Totaro}, Vishik established that, over the field $K=F(x_1,\ldots ,x_5)$, the $16$-dimensional anisotropic form $$q\simeq \<\!\<-x_1,-x_2,-x_3\>\!\>\perp x_4\< 1,x_1,x_2,x_3\>\perp x_5\< 1,x_1,x_2,x_3\>$$ satisfies $i_1(q)=2$ but is not a multiple of a $1$-fold Pfister form. Over the field $L=K(\!(y_1)\!)\ldots (\!(y_n)\!)$ for $n\in\mathbb{N}\cup\{ 0\}$, let $\pi\simeq \<\!\< -y_1,\ldots ,-y_n\>\!\>$ and consider the $2^{n+4}$-dimensional form $q\otimes\pi$. As $i_1(q)=2$, it follows that $i_1(q\otimes\pi)=2^{n+1}$, in accordance with Proposition~\ref{i1trans}. Moreover, as $q$ is not a multiple of a $1$-fold Pfister form over $K$, we have that $q\otimes \pi$ is not a multiple of an $(n+1)$-fold Pfister form over $L$, as follows from $n$ invocations of Proposition~\ref{divtrans}.
\end{example}



\begin{example}\label{vishiksubform} Let $p$ be a $14$-dimensional subform of the form $$q\simeq \<\!\<-x_1,-x_2,-x_3\>\!\>\perp x_4\< 1,x_1,x_2,x_3\>\perp x_5\< 1,x_1,x_2,x_3\>$$ over $K=F(x_1,\ldots ,x_5)$. Consider the $30$-dimensional form $\psi\simeq q\perp yp$ over $K(\!(y)\!)$. As $\psi\subset q\otimes \< 1,y\>$, it follows that $i_1(\psi)=2$, since $i_1(q\otimes \< 1,y\>)=4$ by Example~\ref{vishikform}. Suppose, for the sake of contradiction, that $\psi$ is a multiple of a $1$-fold Pfister form $\rho$ over $K(\!(y)\!)$. Hence $\psi\simeq\rho\otimes\gamma$ for some form $\gamma$ over $K(\!(y)\!)$, with $\gamma\simeq \gamma_1\perp y\gamma_2$ for $\gamma_1$, $\gamma_2$ forms over $K$. As before, we have that $\rho\simeq \< 1,a\>$ or $\rho\simeq \< 1,ay\>$ for some $a\in K^{\times}$. For $\rho\simeq \< 1,a\>$, it follows that $q\simeq \< 1,a\>\otimes \gamma_1$, in contradiction to \cite[Lemma 7.1]{Totaro}. For $\rho\simeq \< 1,ay\>$, it follows that $q\simeq\gamma_1\perp a\gamma_2$ and $p\simeq \gamma_2\perp a\gamma_1$, in contradiction to the fact that $\dim p\neq\dim q$. Hence, $\psi$ is a $30$-dimensional form over $K(\!(y)\!)$ satisfying $i_1(\psi)=2$ that is not a multiple of a $1$-fold Pfister form. By iterating this argument, it follows that, for any $n\in\mathbb{N}$, there exists a form $\varphi$ of dimension $2^{n+4}-2^{n}$ satisfying $i_1(\varphi)=2$ that is not a multiple of a $1$-fold Pfister form.
\end{example}




%
%

\begin{question}\label{vishikq} For each $n\geq 4$, what is the least positive integer $k(n)$ such that there exists a form $q$ of dimension $2^n+k(n)$ satisfying $i_1(q)=2$ that is not a multiple of a $1$-fold Pfister form?
\end{question}

In accordance with Example~\ref{vishiksubform}, we have that $k(n)\leqs 2^n-2^{n-3}$ for every $n\geq 4$. In light of this, it would be interesting to determine whether there exists a form $q$, satisfying the dimension condition $2^n< \dim q < 2^{n+1}-2^{n-3}$ for some $n\geq 4$, that has non-trivial first Witt index but is not isotropy equivalent to a multiple of a Pfister form.

{\emph{Acknowledgements.} I gratefully acknowledge the support I received through an International Mobility Fellowship from the Irish Research Council, co-funded by Marie Curie Actions under FP7. I thank Sylvain Roussey for making his comprehensive PhD thesis available to me.


\begin{thebibliography}{22}
\bibitem[1]{EKM} {R. Elman, N. A. Karpenko, A. S. Merkurjev},
The algebraic and geometric theory of quadratic forms, American Mathematical Society Colloquium Publications {\boldmath $56$}, 
{\em American Mathematical Society} (2008).
\bibitem[2]{EL} {R. Elman, T. Y. Lam}, Pfister forms and K-theory of fields, {\em Journal of Algebra} {\boldmath $23$}, $181-213$ (1972).
\bibitem[3]{ELW} {R. Elman, T. Y. Lam, A. R. Wadsworth}, Orderings under field extensions,
{\em Journal f\"ur die reine und angewandte Mathematik} {\boldmath $306$}, $7-27$ (1979).
\bibitem[4]{F2} {R. W. Fitzgerald}, Witt kernels of function field extensions,
{\em Pacific Journal of Mathematics} {\boldmath $109$}, $89-106$ (1983).
\bibitem[5]{GS} {E. R. Gentile, D. B. Shapiro}, Conservative quadratic forms,
{\em Mathematische Zeitschrift} {\boldmath $163$}, $15-23$ (1978).
\bibitem[6]{H} {D. W. Hoffmann}, Isotropy of quadratic forms over the function field of a quadric,
{\em Mathematische Zeitschrift} {\boldmath $220$}, No. 3, $461-476$ (1995).
\bibitem[7]{H4} {D. W. Hoffmann}, Twisted Pfister Forms, {\em Documenta Mathematica} {\bf 1}, $67-102$ (1996).

\bibitem[8]{Hspl} {D. W. Hoffmann}, Splitting patterns and invariants of quadratic forms, {\em Mathematische Nachrichten} {\bf 190}, $149-168$ (1998).
\bibitem[9]{Imax} {O. T. Izhboldin}, Quadratic forms with maximal splitting, {\em Algebra i Analiz} {\bf 9}, No. 2, $51-57$ (Russian). English translation: {\em St. Petersburg Math Journal} {\bf 9}, No. 2, $219-224$ (1998).
\bibitem[10]{Ilow} {O. T. Izhboldin}, On the isotropy of low-dimensional forms over the function field of a quadric, {\em Algebra i Analiz} {\bf 12}, No. 5, $106-127$ (Russian). English translation: {\em St. Petersburg Math Journal} {\bf 12}, No. 5, $791-805$ (2001).
\bibitem[11]{I} {O. T. Izhboldin}, Fields of $u$-invariant $9$, {\em Annals of Mathematics} (2) {\boldmath $154$}, No. 3, $529-587$ (2001).
\bibitem[12]{IV} {O. T. Izhboldin, A. Vishik}, Quadratic forms with absolutely maximal splitting, Quadratic Forms and Their Applications, Dublin, 1999, {\em Contemporary Mathematics} {\boldmath 272}, $103-125$ (2000).
\bibitem[13]{IKKV} {O. T. Izhboldin, B. Kahn, N. A. Karpenko, A. Vishik}, Geometric methods in the algebraic theory of quadratic forms, Proceedings of the Summer School held at the Universite d'Artois, Lens, June 2000 (ed. J.-P. Tignol) {\em Lecture Notes in Mathematics} {\boldmath $1835$} (2004).
\bibitem[14]{KM} {N. A. Karpenko, A. S. Merkurjev}, Essential dimension of quadrics, {\em Inventiones Mathematicae} {\boldmath $153$},
$361-372$ (2003).
\bibitem[15]{Kn1} {M. Knebusch}, Generic Splitting of Quadratic Forms, I, {\em Proceedings of the London Mathematical Society} $(3)$
{\boldmath  $33$}, $65-93$ (1976).
\bibitem[16]{Kn2} {M. Knebusch}, Generic Splitting of Quadratic Forms, II, {\em Proceedings of the London Mathematical Society} $(3)$
{\boldmath  $34$}, $1-31$ (1977).
\bibitem[17]{Lag} {A. Laghribi}, Sur le probl\`eme de descente des formes quadratiques, {\em Archiv der Mathematik} {\boldmath $73$}, $18-24$ (1999).
\bibitem[18]{LAM} {T. Y. Lam},
Introduction to Quadratic Forms over Fields, {\em American Mathematical Society}
(2005).
\bibitem[19]{OS1} {J. O'Shea}, The weak isotropy of quadratic forms over field extensions, {\em Manuscripta Mathematica} {\boldmath $145$}, $143-161$ (2014).
\bibitem[20]{R} {S. Roussey}, Isotropie, corps de fonctions et \'equivalences birationnelles des
formes quadratiques, PhD thesis, Universit\'e de Franche-Comt\'e (2005).
\bibitem[21]{Totaro} {B. Totaro}, Birational geometry of quadrics, {\em Bull. Soc. Math. France} {\bf 137}, No. 2, 253--276 (2009).
\bibitem[22]{W} {A. R. Wadsworth}, Noetherian pairs and function fields of quadratic forms, PhD thesis, University of Chicago (1972).
\bibitem[23]{WS} {A. R. Wadsworth, D. B. Shapiro}, On multiples of round and Pfister forms, {\em Mathematische Zeitschrift} {\boldmath $157$}, $53-62$ (1977).
\end{thebibliography}
\end{document}